\documentclass[a4paper,12pt]{article}
\usepackage[sumlimits,intlimits]{amsmath}
\usepackage{amsthm,amssymb,amsopn}
\usepackage[top=1in,bottom=1in,left=1in,right=1in]{geometry}

\usepackage{txfonts}

\allowdisplaybreaks

\begin{document}

\newtheorem{definition}{Definition}[section]
\newtheorem{proposition}[definition]{Proposition}
\newtheorem{lemma}[definition]{Lemma}
\newtheorem{theorem}[definition]{Theorem}
\newtheorem{corollary}[definition]{Corollary}
\newtheorem{remark}[definition]{Remark}
\newtheorem*{remark*}{Remark}

\theoremstyle{definition}
\newtheorem*{myproof}{Proof}

\numberwithin{equation}{section}

%%%%%%%%%%%%%%%%%%%%%%%%%%%%%%%%%%%%%%%%%%%

\renewcommand{\a}{\alpha}
\newcommand{\Acal}{\mathcal{A}}
\newcommand{\ds}{\displaystyle}
\newcommand{\hF}{\hat{F}}
\newcommand{\ga}{\gamma}
\newcommand{\Ga}{\Gamma}
\newcommand{\hG}{\hat\Gamma}
\newcommand{\p}{\partial}
\newcommand{\pf}{\medskip \noindent {\sl Proof}. \ }
\newcommand{\Pcal}{\mathcal{P}}
\newcommand{\RR}{\mathbb{R}}
\newcommand{\ve}{\varepsilon}
\newcommand{\vp}{\varphi}
\newcommand{\tA}{\tilde A}
\newcommand{\tB}{\tilde B}
%%%%%%%%%%%%%%%%%%%%%%%%%%%%%%%%%%%%%%%%%%%

%\newcommand{\abs}[1]{\lvert#1\rvert}
%\newcommand{\norm}[1]{\lVert#1\rVert}
%\newcommand{\wt}[1]{\widetilde{#1}}
%%%%%%%%%%%%%%%%%%%%%%%%%%%%%%%%%%%%%%%
%\begin{figure}[htbp]
%\centering
%\includegraphics[scale=.7]{***.eps}
%\caption{*****}\label{*****}
%\end{figure}
%%%%%%%%%%%%%%%%%%%%%%%%%%%%%%%%%%%%%%%

\title{Estimate of the Fundamental Solution for Parabolic Operators with Discontinuous Coefficients}
\author{Michele Di Cristo\thanks{Department of Mathematics, Polytechnic University of Milan, Milan 20133, Italy. This work was supported by JSPS Postdoctoral Fellowship for Foreign Resaerchers (PE 08039)}\and Kyoungsun Kim\thanks{Department of Mathematics, Ehwa Womans University, Seoul 120-750, Korea. This work was
supported by the Korea Research Foundation Grant funded by the
Korean Government(MOEHRD)(KRF-2006-214-C00007).}\and Gen
Nakamura\thanks{Department of Mathematics, Hokkaido University, Sapporo 080-061, Japan. This work was partially supported by Grant-in-Aid
for Scientific Research (B)(No. 19340028) of Japan Society for
Promotion of Science.}}
\date{}

\maketitle
%$\imath$

\begin{center}\begin{minipage}[h]{.8\textwidth}\small {\bf Abstract.} We
will show that the same type of estimates known for the fundamental solutions for scalar parabolic equations with smooth enough coefficients hold for the first order derivatives of fundamental solution with respect to space variables of scalar parabolic equations of divergence form with discontinuous coefficients. The estimate is very important for many applications. For example, it is important for the inverse problem identifying inclusions inside a heat conductive medium from boundary measurements.

\medskip

\end{minipage}\end{center}
\section{Introduction}\label{sec::introduction}
Let $\mathcal{L}$ be a parabolic operator of the form
\begin{equation}\label{eqn::operator}
\mathcal{L}=\partial_t - \nabla \cdot A \nabla
\end{equation}
with an $n\times n$ matrix $A=(a_{ij})\in L^{\infty}(D)$ and a
bounded domain $D\subset \mathbb{R}^n$ with boundary $\partial D$ of
Lipschitz class. $D$ has compactly embedded subdomain
$D_m\,(m=1,2,\cdots, L)$ with boundaries $\partial
D_m\,(m=1,2,\cdots,L)$ of $C^{1,\alpha}$ class such that $\overline
D_\ell\cap\overline D_m=\emptyset\,(\ell\not=m)$,
$\overline{D_m}\subset D$ $(1\leq m\leq L)$, where $0<\alpha<1$.
Moreover, there exists a constant $\delta>0$ such that
\begin{equation}\label{eqn::positivity}
\sum_{i,j=1}^{n} a_{ij}(x)\xi_i\xi_j\geq \delta \sum_{i=1}^n\xi_i^2\quad (\xi\in \mathbb{R}^n,\, \mbox{a.e.}\,x\in D).
\end{equation}
Also, we assume that $A\in C^{\mu}(\overline{D_m})$ ($1\leq m \leq L$) with $0<\mu<1$.

We want to show the following theorem which is quite important in
inverse problems for heat equations with discontinuous
coefficients.(See the remark after the following theorem.)

\begin{theorem}\label{estimate of gamma}
There exists a fundamental solution $\Gamma(x,t;y,s)$ of $\mathcal{L}$ with the estimates
\begin{equation}\label{estimate}
0 < \Gamma(x,t;y,s) \leq \frac{C}{(t-s)^{n/2}}e^{-\frac{|x-y|^2}{C(t-s)}},
\end{equation}
\begin{equation}\label{gradient estimate}|\nabla_x \Gamma(x,t;y,s)|\leq
\frac{C}{(t-s)^{(n+1)/2}}e^{-\frac{|x-y|^2}{C(t-s)}}
\end{equation}
for any $t,s\in{\Bbb R}$, $t>s$ and almost every $x,y\in D$, where
$C>0$ is a constant depending only on $A$ and $n$.
\end{theorem}

\begin{remark}
${}$
\newline
(i) For the simplicity of notations, we confined our argument to
scalar parabolic operators of divergence form without zeroth order
term. However, our argument can be generalized not only to more
general scalar parabolic operators, but also to parabolic systems.
\newline
(ii) The estimate \eqref{estimate} is the well known estimate of the
fundamental solution (\cite{A}). We will call the estimate
\eqref{gradient estimate} {\rm gradient estimate}. This gradient
estimate is quite crucial for the dynamical probe method \cite{IKN}
and stability estimate for the inverse boundary value problem
\cite{D-V}. Here, the dynamical probe method is a reconstruction
scheme for the inverse boundary value problem identifying an unknown
discontinuities of a medium inside a known heat conductor from our
measurements on the boundary which are so called Dirichlet to
Neumann map or the Neumann to Dirichlet map. The graphs of these
maps are nothing but the set of infinitely many Cauchy data of the
solutions to the forward problem of this inverse problem. Also, the
stability estimate is the estimate of continuity of the map which
maps the set of Cauchy data to the unknown this continuity of media.
\end{remark}

We will show later that the gradient estimate of the fundamental
solution follows from the following interior estimate following the
argument given in \cite{D-V}.

\begin{theorem}[Main Theorem]\label{thm::main}
Let $0<r<T$, $\overline{r\Omega}\subset D$ and $u\in W(r\Omega\times
(-r^2,r^2))$; $u\in L^2((-r^2,r^2);H_0^1(r\Omega))$; $\partial_t u
\in L^2((-r^2,r^2);H^{-1}(r\Omega))$ be a solution of $(\partial_t -
\nabla \cdot A \nabla)u=0$ in $r\Omega \times (-r^2,r^2)$. Then,
there exists a constant $c>0$ such that for any $0<\rho <r/2$ and
$(x,t)\in (r-2\rho)\Omega \times (-r^2+4\rho^2,r^2)$, we have
\begin{equation}\label{eqn::es1}
\|\nabla_x u\|_{L^{\infty}(\rho\Omega(x)\times (-\rho^2+t,t))}\leq \frac{c}{\rho^{n/2+2}}\|u\|_{L^2(2\rho\Omega(x)\times (-4\rho^2+t,t))},
\end{equation}
where $\Omega(x):=\{y=(y_1,\cdots,y_n)\in \mathbb{R}^n;|x_i-y_i|<1\,(1\le i\le n)\}$ with $x=(x_1,\cdots,x_n)$ and $\Omega:=\Omega(0)$.
\end{theorem}

\begin{remark} In \cite{D-V} there is also a statement of the corresponding main theorem. Here we will give a proof of the main theorem based on the argument of \cite{L-N} for second order elliptic systems of divergence forms which could provide a proof of the main theorem even in the case the inclusions touch.
\end{remark}

By translation, rotation and scaling, it is enough to prove the following.

\begin{proposition}\label{prop}
Let $A$ as in Theorem \ref{thm::main}. Then, for any solution $u\in W(\Omega\times (-1,1))$ of $(\partial_t - \nabla \cdot A \nabla)u=0$ in $\Omega \times (-1,1)$, there exists a constant $c>0$ depending only on the bound of $A$ and $n$ such that
\begin{equation}\label{eqn::es2}
|\nabla_x u(0,\tau)|\leq c\|u\|_{L^2(\frac{1}{2}\Omega\times (\frac{3}{4}\tau-\frac{1}{4},\tau))}\quad (\tau \in (\frac{1}{3},1))
\end{equation}
whenever $\nabla_x u(0,\tau)$ exists.
\end{proposition}

The rest of this paper is organized as follows. In the next section we will give the proof of Proposition \ref{prop} by assuming the existence of a Green function in a two layered cube with Dirichlet boundary condition and its estimates. The proof of the gradient estimate of a fundamental solution is given in section 3. Finally, in Appendix, we give a construction and estimates of the Green function in the two layered cube with Dirichlet boundary condition.

\section{Proof of Proposition \ref{prop}}\label{sec::proof}
We will adapt the proof of Li-Nireberg's paper \cite{L-N}. To begin with, we note that Lemma 4.3 in \cite{L-N} also holds if $\|\cdot\|_{Y^{1+\alpha',2}}$ is replaced by $\|\cdot\|_{Y^{1+\alpha',p}}$ for any $p>1$. That is
\begin{lemma}[Lemma 4.3']\label{lem::4.3m}
For $0<\alpha'\leq \min\{\mu,\frac{\alpha}{p(\alpha+1)}\}$, there exists a constant $E>0$ depending only on $\|A\|_{C^{\alpha'}(D_m)}$ for these $D_m$'s which intersect with $\Omega$ such that
\begin{equation}\label{eqn::es3}
\|A-\overline{A}\|_{Y^{1+\alpha',p}}\leq E,
\end{equation}
$\overline{A}$ is defined as in \cite{L-N} for the very special case i.e. the two layered cube $\Omega$.
\end{lemma}
\begin{proof}
This can be easily proved by checking the proof in \cite{L-N}.
\end{proof}

\begin{remark}
Hereafter, names of theorems in parentheses correspond to those of the theorem in \cite{L-N}. For example, Lemma \ref{lem::4.3m}(Lemma 4.3') correspond
to Lemma 4.3 in \cite{L-N}.
\end{remark}

Next we generalize Lemma 3.1 in \cite{L-N} in a special way to the parabolic operator \eqref{eqn::operator}. That is we have the following:
\begin{lemma}[Lemma 3.1']\label{lem::3.1m}
For $0<\varepsilon<1$, let the previous $A$ satisfy
\begin{equation}\label{eqn::es3m}
\|A-\overline{A}\|_{Y^{1+\alpha',p}}\leq \varepsilon,
\end{equation}
where $p$ will be specified in the proof. Here, we note that we have properly scaled $A$ so that by applying Lemma \ref{lem::4.3m} (Lemma 4.3'), \eqref{eqn::es3m} is satisfied. Then, for any $f\in L^2((-1,\tau);L^2(\Omega))$ and solution $u\in W(\Omega\times (-1,\tau))$ of $$(\partial_t - \nabla \cdot A \nabla)u=\nabla \cdot f\quad \mbox{in}\quad \Omega \times (-1,\tau),$$
there exists a solution $v\in W(\frac{3}{4}\Omega\times (\frac{7}{16}\tau-\frac{9}{16},\tau))$ of
$$(\partial_t - \nabla \cdot\overline A \nabla)v=0\quad \mbox{in}\quad \frac{3}{4}\Omega\times (\frac{7}{16}\tau-\frac{9}{16},\tau)$$
and $\frac{3}{4} <\sigma_0<1$ such that for any $0<\sigma \leq \sigma_0$
\begin{equation}\label{eqn::es4}
\|u-v\|_{L^2(\sigma\Omega\times((1-\sigma^2)\tau-\sigma^2,\tau))}\leq C\Big(\|f\|_{L^2(\Omega\times (-1,\tau))}+ \varepsilon^{1/2}\|u\|_{L^2(\Omega\times (-1,\tau))}\Big),
\end{equation}
where $C>0$ depends only on $n$ and $A$.
\end{lemma}
\begin{proof}
It is enough to prove the estimate for the case that $\Omega$ is divided into two parts by the boundary $\partial D_m$ of one of $D_m\,(1\le m\le L)$ such that the center $0\in\partial D_m$ of $\Omega$ and $\Omega$ does not contain any portion of $\partial D_\ell\,(\ell\not=m,\,1\le\ell\le m)$.

%For simplicity, we denote $Q:=\Omega \times (-1,\tau)$, $Q_1:=\frac{1}{2}\Omega \times %(\frac{3}{4}\tau-\frac{1}{4},\tau)$ and $Q_2:=\frac{3}{4}\Omega \times %(\frac{7}{16}\tau-\frac{9}{16},\tau)$.
Take a cutoff function $\zeta\in C_0^{\infty}(\Omega \times (-1,2\tau))$. By the definition of weak solution, we have
$$\int_{-1}^{\tau}\Big \{(\partial_t u,\varphi)+(A\nabla u,\nabla \varphi)\Big \}dt =-\int_{-1}^{\tau}(f,\nabla \varphi)dt$$
for any $\varphi \in H_0^1(\Omega)$.

If we take $\varphi=\zeta^2u$, then we have
$$\int_{-1}^{\tau}\Big \{(\partial_t u,\zeta^2u)+(A\nabla u,\nabla (\zeta^2u))\Big \}dt =-\int_{-1}^{\tau}(f,\nabla (\zeta^2u))dt.$$
Here, we have
$$
\begin{aligned}
\mbox{LHS}\geq& \frac{1}{2}\int_{\Omega} (\zeta^2u^2)(\cdot,\tau)dx+\int_{-1}^{\tau}\int_{\Omega}\zeta \zeta_tu^2dxdt\\
&+\delta \int_{-1}^{\tau}\int_{\Omega}\zeta^2|\nabla u|^2 dx dt-C\int_{-1}^{\tau}\int_{\Omega}(\zeta |\nabla u|)(u |\nabla \zeta|)dx dt,
\end{aligned}$$
where $C>0$ denotes any general constant. We further have
$$\mbox{LHS}\geq \frac{\delta}{2} \int_{-1}^{\tau}\int_{\Omega}\zeta^2|\nabla u|^2 dx d\tau-C(\delta)\int_{-1}^{\tau}\int_{\Omega}u^2dx d\tau,$$
where $C(\delta)>0$ denotes any general constant depending on $\delta$.

On the other hand, we have
$$
\begin{aligned}
\mbox{RHS}=& -\int_{-1}^{\tau}\int_{\Omega}(\zeta f \cdot (\zeta \nabla u)+uf \cdot \nabla \zeta^2)dxd\tau\\
&\leq \frac{\delta}{4} \int_{-1}^{\tau}\int_{\Omega}\zeta^2|\nabla u|^2 dx d\tau+C(\delta)\int_{-1}^{\tau}\int_{\Omega}(u^2+ |f|^2)dx d\tau,
\end{aligned}$$
Hence,
$$
\frac{\delta}{4} \int_{-1}^{\tau}\int_{\Omega}\zeta^2|\nabla u|^2 dx d\tau\leq C(\delta)\int_{-1}^{\tau}\int_{\Omega}(u^2+ |f|^2)dx d\tau.$$
By $|\nabla (\zeta u)|^2\leq 2\zeta^2|\nabla u|^2+2u^2|\nabla \zeta|^2$, we have
$$
\frac{\delta}{8} \int_{-1}^{\tau}\int_{\Omega}|\nabla (\zeta u)|^2 dx d\tau\leq C(\delta)\int_{-1}^{\tau}\int_{\Omega}(u^2+ |f|^2)dx d\tau.$$
Then, combining this with Poincare's inequality, we have
$$\|\zeta u\|_{L^2((-1,\tau);H^1(\Omega))}\leq C(\delta)\Big (\|u\|_{L^2(\Omega\times (-1,\tau))}+\|f\|_{L^2(\Omega\times (-1,\tau))}\Big ).$$
If we let $\zeta$ to satisfy $\zeta =1$ in a neighborhood of $\overline{\frac{4}{5}\Omega\times (\frac{9}{25}\tau-\frac{16}{25},\tau)}$, we have
\begin{equation}\label{eqn::es5}
\|u\|_{L^2( (\frac{9}{25}\tau-\frac{16}{25},\tau);H^1(\frac{4}{5}\Omega))}\leq C(\delta)\Big (\|u\|_{L^2(\Omega\times (-1,\tau))}+\|f\|_{L^2(\Omega\times (-1,\tau))}\Big ).
\end{equation}
By \eqref{eqn::es5} and the Fubini theorem, there exists $\frac{3}{4}\leq \sigma_0<1$ such that
$$\|u\|_{L^2(((1-\sigma_0^2)\tau-\sigma_0^2,\tau);H^1(\partial(\sigma_0\Omega)))}\leq C(\delta)\Big (\|u\|_{L^2(\Omega\times (-1,\tau))}+\|f\|_{L^2(\Omega\times (-1,\tau))}\Big )$$

Let $v\in W(\sigma \Omega \times (-1,\tau))$ be the solution to
$$\left \{\begin{aligned}
&(\partial_t - \nabla\cdot\overline{A}\nabla )v=0 &&\mbox{in}\quad \sigma_0\Omega \times ((1-\sigma_0^2)\tau-\sigma_0^2,\tau),\\
&v=u &&\mbox{on}\quad  \partial (\sigma_0 \Omega) \times ((1-\sigma_0^2)\tau-\sigma_0^2,\tau),\\
&v=u &&\mbox{on}\quad  \sigma_0 \Omega \times \{(1-\sigma_0^2)\tau-\sigma_0^2\}.
\end{aligned}\right..$$

%{\color{blue}(check the existence of trace $u|_{\sigma\Omega \times\{ (1-\sigma^2)\tau-\sigma^2\}}$).}
Then, $w:=u-v$ satisfies
$$\left \{\begin{aligned}
&(\partial_t - \nabla\cdot\overline{A}\nabla )w=\nabla \cdot ((A-\overline{A})\nabla u)+\nabla \cdot f &&\mbox{in}\quad \sigma_0\Omega \times ((1-\sigma_0^2)\tau-\sigma_0^2,\tau),\\
&w=0 &&\mbox{on}\quad  \partial (\sigma_0 \Omega) \times ((1-\sigma_0^2)\tau-\sigma_0^2,\tau),\\
&w=0 &&\mbox{on}\quad  \sigma_0 \Omega \times \{(1-\sigma_0^2)\tau-\sigma_0^2\}.
\end{aligned}\right.$$
Further, let $G^*(x,t;y,s)$ be the Green function such that for $y\in \sigma\Omega$ and $s \in \mathbb{R}$. That is, $G^*(x,t;y,s)$ is a distribution which satisfies
$$\left \{\begin{aligned}
&(\partial_t +\nabla\cdot\overline{A}\nabla )G^*(x,t;y,s)=0 &&\mbox{in}\quad \sigma_0\Omega \times (-\infty, s),\\
&G^*(x,t;y,s)=0 &&\mbox{on}\quad  \partial (\sigma_0 \Omega) \times (-\infty,s),\\
&\lim_{t\uparrow
s}\int_{\sigma_0\Omega}G^*(x,t;y,s)\varphi(x)dx=\varphi(y)
&&(\varphi\in C_0^{\infty}(\sigma_0\Omega)).
\end{aligned}\right.$$
Later in the appendix, we will provide the construction of $G^*(x,t;y,s)$ and proof of its estimate:
\begin{equation}\label{eqn::esGreen}
|\partial_x^{\alpha}G^*(x,t;y,s)|\leq c_{\alpha}(s-t)^{-\frac{n+|\alpha|}{2}}e^{-\frac{|x-y|^2}{c(s-t)}}\quad (t,s\in \mathbb{R},\,t<s,\,\mbox{a.e.}\,x,y\in \sigma_0 \Omega).
\end{equation}
for any $\alpha\in \mathbb{Z}_+^n,\ |\alpha|\le1$ with some constant $c_{\alpha}>0$.

By the Green formula, we have
\begin{equation}\label{eqn::Greenf}
w(x,t)=-\int_{(1-\sigma_0^2)\tau-\sigma_0^2}^{\tau}\int_{\sigma_0\Omega}\Big\{((A-\overline{A})\nabla u)(y,s)+f(y,s)\Big\}\cdot \nabla_y G^*(y,s;x,t)dyds.
\end{equation}
for $x\in \sigma_0\Omega$, $t\in ((1-\sigma_0^2)\tau-\sigma_0^2,\tau)$.

To proceed further, we need the following Lemma \ref{lem::fourier} which is well-known in Fourier analysis.

\begin{lemma}\label{lem::fourier}
Let $(X_i,M_i,m_i)$ $(i=1,2)$ be $\sigma$ finite complete measure spaces and $(X_1\times X_2, M_1\otimes M_2, m_1\times m_2)$ be the product space with complete measure $m_1\times m_2$. Also, let $p_1$, $p_2$ and $q\in [1,\infty)$ satisfy $1/p_2+1/q=1/p_1+1$ and a measurable function $K(x_1,x_2)$ on $X_1\times X_2$ satisfy
$$\int_{X_1} |K(x_1,x_2)|^q m_1(dx_1)\leq L_1\quad (\mbox{a.e.}\,x_2 \in X_2),$$
$$\int_{X_2} |K(x_1,x_2)|^q m_2(dx_2)\leq L_2\quad (\mbox{a.e.}\,x_1 \in X_1).$$
Then, for any $f(x_2)\in L^{p_2}(X_2)$, we have
$$\|Kf\|_{ L^{p_1}(X_1)}\leq L_1^{1/p_1}L_2^{1-1/p_2}\Vert f\Vert_{L^{p_2}(X_2)},$$
where
$$(Kf)(x_1)=\int_{X_2}K(x_1,x_2)f(x_2)m_2(dx_2).$$
\end{lemma}

We apply this Lemma \ref{lem::fourier} to the kernel $(A-\overline{A})(y,s)\nabla_yG^*(y,s;x,t)$ ($(x,t),\,(y,s)\in \sigma_0 \Omega \times ((1-\sigma_0^2)\tau-\sigma_0^2,\tau),\,s<t$) by taking $q=1$, $p_1=p_2=2$ and
$X_1=X_2={\sigma}_0\Omega\times((1-{\sigma}_0^2)\tau-{\sigma}_0^2,\tau)$.

Let
$$I_1(y,s)=\int_{(1-{\sigma}_0^2)\tau-{\sigma}_0^2}^{\tau}\int_{{\sigma}_0\Omega}
\Big|(A-\overline{A})(y,s) \nabla_y G^*(y,s;x,t)\Big|dxdt,$$
and
$$I_2(x,t)=\int_{(1-\sigma^2)\tau-\sigma^2}^{\tau}\int_{\sigma_0\Omega}
\Big|(A-\overline{A})(y,s) \nabla_y G^*(y,s;x,t)\Big|dyds.$$ Since
$A$, $\overline{A}$ are bounded and
$\int_{\mathbb{R}^n}e^{-\frac{|x-y|^2}{c(t-s)}}dx=O((t-s)^{n/2})$
$(s<t)$,
$$I_1(y,s)\leq C\int_s^{\tau}(t-s)^{-1/2}dt =2C(\tau-s)^{1/2}\leq C.$$
Similarly, by taking $p>n+2$,
$$\begin{aligned}
I_2(x,t)\leq &\Bigg(\int_{(1-\sigma_0^2)\tau-\sigma_0^2}^{\tau}\int_{\sigma_0\Omega}
\Big|(A-\overline{A})(y,s) \Big|^{p}dyds\Bigg)^{1/p}\Bigg(\int_{(1-\sigma_0^2)\tau-\sigma_0^2}^{\tau}
\int_{\sigma_0\Omega}
\Big|\nabla_y G^*(y,s;x,t)\Big|^{p^*}dyds\Bigg)^{1/p^*}\\
&\leq
C(\tau-(1-\sigma_0^2)\tau+\sigma_0^2)^{1/p}\|A-\overline{A}\|_{L^p(\Omega)}\Bigg
(\int_{(1-\sigma_0^2)\tau-\sigma_0^2}^{\tau}(t-s)^{n/2-(n+1)p^{*}/2}ds\Bigg
)^{1/p^*}.
\end{aligned}$$
with $1/p+1/p^*=1$. Here, by $p>n+2$, $n/2-(n+1)p^{*}/2>-1$. Hence, $I_2(x,t)\leq C\|A-\overline{A}\|_{L^p(\Omega)}$. Therefore, by Lemma \ref{lem::fourier} and \eqref{eqn::es5},
$$
\|w\|_{L^2({\sigma}_0 \Omega \times ((1-{\sigma}_0^2)\tau-{\sigma}_0^2,\tau))}\leq C\Big(\|A-\overline{A}\|^{1/2}_{L^p(\Omega)}\|\nabla u\|_{L^2(\sigma_0 \Omega \times ((1-\sigma_0^2)\tau-\sigma_0^2,\tau))}+\|f\|_{L^2(\sigma_0 \Omega \times ((1-\sigma_0^2)\tau-\sigma_0^2,\tau))}\Big).$$
Since $\|A-\overline{A}\|_{Y^{1+\alpha',p}}< \varepsilon$ implies $\|A-\overline{A}\|_{L^p(\Omega)}< C\varepsilon$, we have
$$
\|w\|_{L^2({\sigma}_0 \Omega \times
((1-{\sigma}_0^2)\tau-{\sigma}_0^2,\tau))}\leq
C\Big(\varepsilon^{1/2}\|u\|_{L^2(\Omega \times
(-1,\tau))}+\|f\|_{L^2(\Omega \times (-1,\tau))}\Big).$$ Further, by
applying \eqref{eqn::es5} to $w$, we have for a smaller $\sigma_0$
$$
\|w\|_{L^2(((1-\sigma^2)\tau-\sigma^2,\tau);H^1(\sigma \Omega))}\leq
C\Big(\varepsilon^{1/2}\|u\|_{L^2(\Omega \times
(-1,\tau))}+\|f\|_{L^2(\Omega \times (-1,\tau))}\Big)\quad
(0<\sigma\leq \sigma_0).$$ This ends the proof.
\end{proof}

\begin{proof}[Proof of \eqref{eqn::es2}]
We adapt the proof of Proposition 4.1 in \cite{L-N} to our case.

Let $\sigma_k=\frac{3}{4^{k+1}}$, $\overline{\sigma}_k=\frac{2}{4^{k+1}}$, $\widetilde{\sigma}_k=\frac{1}{4^{k+1}}$ ($k=0,1,2,\cdots$) and $M=\|u\|_{L^2(\frac{1}{2}\Omega \times (\frac{3}{4}\tau-\frac{1}{4},\tau))}$. We will prove by induction that there exist $w_k\in W(\sigma_k\Omega \times ((1-\sigma_k^2)\tau-\sigma_k^2,\tau))$ ($k=0,1,2,\cdots$) which satisfy
$$(\partial_t -\nabla \cdot \overline{A}\nabla)w_k=0\quad \mbox{in}\quad \sigma_k\Omega \times ((1-\sigma_k^2)\tau-\sigma_k^2,\tau),\eqno{(4.3)'_k}$$
$$\|w_k\|_{L^2(\overline{\sigma}_k\Omega \times ((1-\overline{\sigma}_k^2)\tau-\overline{\sigma}_k^2,\tau))}\leq CM4^{-\frac{k(n+4+2\alpha')}{2}},\eqno{(4.4)'_{1,k}}$$
$$\|\nabla_x w_k\|_{L^{\infty}(\widetilde{\sigma}_k\Omega \times ((1-\widetilde{\sigma}_k^2)\tau-\widetilde{\sigma}_k^2,\tau))}\leq CM4^{-k\alpha'},\eqno{(4.4)'_{2,k}}$$
and
$$\Bigg \|u-\sum_{j=0}^{k} w_j\Bigg\|_{L^{\infty}(\widetilde{\sigma}_k\Omega \times ((1-\widetilde{\sigma}_k^2)\tau-\widetilde{\sigma}_k^2,\tau))}\leq M4^{-\frac{(k+1)(n+4+2\alpha')}{2}}.\eqno{(4.5)'_k}$$

Before starting the induction argument, we note that for any $\varepsilon_0>0$, we have $\|A-\overline{A}\|_{Y^{1+\alpha',p}}\leq \varepsilon_0$ by considering a dilated $\Omega$ instead of $\Omega$, and $\|A-\overline{A}\|_{L^p(\Omega)}\leq \|A-\overline{A}\|_{Y^{1+\alpha',p}}$.

Since $u\in W(\frac{1}{2}\Omega \times (\frac{3}{4}\tau-\frac{1}{4},\tau))$ solves $(\partial_t -\nabla \cdot {A}\nabla)u=0$ in $\frac{1}{2}\Omega \times (\frac{3}{4}\tau-\frac{1}{4},\tau)$, we have from Lemma \ref{lem::3.1m} that there exists a solution $w_0\in W(\sigma_0\Omega \times ((1-\sigma_0^2)\tau-\sigma_0^2,\tau))$ of
$(\partial_t -\nabla \cdot \overline{A}\nabla)w_k=0$ in $\sigma_0\Omega \times ((1-\sigma_0^2)\tau-\sigma_0^2,\tau)$ with the estimates:
$$
\Vert
u-w_0\Vert_{L^2(\overline\sigma_0\Omega\times((1-\overline\sigma_0^2)\tau-\overline\sigma_0^2,\tau))}
\le C_0\varepsilon^{1/2} M\le M
4^{-\frac{n+4+2\alpha'}{2}}\eqno{(4.5)'_0}
$$
by taking $\varepsilon_0>0$ small enough to satisfy
$C_0\varepsilon^{1/4}\le1$, $\varepsilon_0^{1/4}\le
4^{-\frac{n+4+2\alpha'}{2}}$ and hence
$$
\Vert w_0\Vert_{L^2(\overline\sigma_0\Omega\times((1-\overline\sigma_0^2)\tau-\overline\sigma_0^2,\tau))}\le
2M,
\eqno{(4.4)'_{1,0}}
$$
and from the interior estimate for $\partial_t-\nabla\cdot\overline A\nabla$, we have
$$
\Vert\nabla_x w_0\Vert_{L^2(\widetilde\sigma\Omega\times((1-\widetilde\sigma_0^2)\tau-\widetilde\sigma_0^2,\tau))}
\le C_0 M.
$$
Hereafter, $C_0>0$ is a general constant for the estimate of solutions for our parabolic operators which is independent of the general constant $C$ in the estimates (4.4)'-(4.5)'.

Suppose (4.3)'-(4.5)' hold up to $k\ge0$. Then, we will prove them for $k+1$. Let
$$
W(x,t)=\Big(u-\sum_{j=0}^k w_j\Big)(\widetilde\sigma_k
x,\widetilde\sigma_k^2 t+(1-\widetilde\sigma_k^2)\tau)\quad((x,t)
\in\Omega\times(-1,\tau)),
$$
$$
A_{k+1}(x)=A(\widetilde\sigma_k x),\,\,\overline A_{k+1}(x)=\overline A(\widetilde\sigma_k x)
$$
and
$$
f_{k+1}(x,t)=\widetilde\sigma_0(A_{k+1}-\overline A_{k+1})(x)
\sum_{j=0}^k\nabla w_j(\widetilde\sigma_k x,\widetilde\sigma_k^2
t+(1-\widetilde\sigma_k^2)\tau).
$$
Then, it is not hard to see that $(\partial_t-\nabla\cdot A_{k+1}\nabla)W=\nabla\cdot f_{k+1}$ in $\Omega\times(-1,\tau)$.
Further, we have
$$
\sum_{j=0}^k|(\nabla_x w_j)(\widetilde\sigma_k
x,\widetilde\sigma_k^2 t+(1-\widetilde\sigma_k^2)\tau)| \le
CM\sum_{j=0}^k
4^{-j\alpha'}\le\frac{CM}{1-4^{-\alpha'}}\,\,((x,t)\in\Omega\times(-1,\tau))
$$
by $(4.4)'_{2,k}$, and
$$
\Vert W\Vert_{L^2(\Omega\times(-1,\tau))}\le M4^{-(k+1)(1+\alpha')}
$$
by $(4.5)'_k$.
Observe that
$$
\Vert A_{k+1}-\overline A_{k+1}\Vert_{L^2(\Omega)}\le 4^{-(k+1)\alpha'}\Vert A-\overline A\Vert_{Y^{1+\alpha',2}}\le 4^{-(k+1)\alpha'}\Vert A-\overline A\Vert_{Y^{1+\alpha',p}}
\le 4^{-(k+1)\alpha'}\varepsilon_0.
$$
Together with this and $(4.4)'_{2,k}$, we have
$$
\Vert f_{k+1}\Vert_{L^2(\Omega\times (-1,\tau))}\le CM4^{-(k+1)(1+\alpha')}\varepsilon_0
\sum_{j=0}^k 4^{-j\alpha'}\le
\frac{CM}{1-4^{-\alpha'}}\varepsilon_0.
$$
By Lemma \ref{lem::3.1m}, there exists a solution $v_{k+1}\in W(\sigma_0\Omega\times((1-\sigma_0^2)\tau-\sigma_0^2,\tau))$ of
$(\partial_t-\nabla\cdot\overline A_{k+1}\nabla) v_{k+1}=0$ in $\sigma_0\Omega
\times((1-\sigma^0)\tau-\sigma_0^2,\tau)$ with the estimate
\begin{equation*}%\displaywidth=-\leftskip
\begin{aligned}
\Vert W-v_{k+1}\Vert_{L^2(\overline\sigma_0\Omega\times ((1-\overline\sigma_0^2)\tau-\overline\sigma_0^2,\tau))}
\le&
C_0(\Vert f_{k+1}\Vert_{L^2(\Omega\times(-1,\tau)}+4^{-(k+1)\alpha'/2}\varepsilon_0^{1/2}
\Vert W\Vert_{L^2(\Omega\times(-1,\tau))})\\
\le&
C_0 M4^{-(k+1)(1+\alpha')}\Bigg (\frac{C}{1-4^{-\alpha'}}\varepsilon_0+\varepsilon_0^{1/2}\Bigg ).
\end{aligned}
\end{equation*}

Let
$w_{k+1}(x,t)=v_{k+1}(\widetilde\sigma_k^{-1}x,\widetilde\sigma_k^{-2}t+(1-\widetilde\sigma_k^{-2})\tau)$
$((x,t)\in\sigma_{k+1}\Omega\times((1-\sigma_{k+1}^2)\tau-\sigma_{k+1}^2,\tau))$.
Then, it is easy to see that $(\partial_t-\nabla\cdot\overline
A\nabla)w_{k+1}=0$ in
$\sigma_{k+1}\Omega\times((1-\sigma_{k+1}^2)\tau-\sigma_{k+1}^2,\tau))$.
Further, by $v_{k+1}(x,t)=w_{k+1}(\widetilde\sigma_k
x,\widetilde\sigma_k^{2}t+(1-\widetilde\sigma_k^{2})\tau)$
($(x,t)\in
\sigma_0\Omega\times((1-\sigma_0^2)\tau-\sigma_0^2,\tau)$),
$$(W-v_{k+1})(x,t)=\Big (u-\sum_{j=0}^{k+1}w_j\Big )(\widetilde\sigma_k x,\widetilde\sigma_k^{2}t+(1-\widetilde\sigma_k^{2})\tau),$$
i.e.
$$\Big(u-\sum_{j=0}^{k+1}w_j\Big)(x,t)=(W-v_{k+1})(\widetilde\sigma_k^{-1}x,\widetilde\sigma_k^{-2}t+(1-\widetilde\sigma_k^{-2})\tau).$$

Since $\widetilde\sigma_0\Omega\times((1-\widetilde\sigma_0^2)\tau-\widetilde\sigma_0^2,\tau) \subset \overline\sigma_0\Omega\times((1-\overline\sigma_0^2)\tau-\overline\sigma_0^2,\tau)$ and $(x,t)\in \widetilde\sigma_{k+1}\Omega\times((1-\widetilde\sigma_{k+1}^2)\tau-\widetilde\sigma_{k+1}^2,\tau)$ which equivalents to $(\widetilde\sigma_k^{-1}x,\widetilde\sigma_k^{-2}t+(1-\widetilde\sigma_k^{-2})\tau)\in \widetilde\sigma_0\Omega\times((1-\widetilde\sigma_0^2)\tau-\widetilde\sigma_0^2,\tau)$, we have
$$\Big\|u-\sum_{j=0}^{k+1}w_j\Big\|_{L^2(\widetilde\sigma_{k+1}\Omega\times((1-\widetilde\sigma_{k+1}^2)\tau-\widetilde\sigma_{k+1}^2,\tau))}=\widetilde\sigma_k^{-(n+2)/2}\|W-v_{k+1}\|_{L^2(\widetilde\sigma_0\Omega\times((1-\widetilde\sigma_0^2)\tau-\widetilde\sigma_0^2,\tau))}$$
and hence by the estimate of $\|W-v_{k+1}\|_{L^2(\overline\sigma_0\Omega\times((1-\overline\sigma_0^2)\tau-\overline\sigma_0^2,\tau))}$, we have
$$
\begin{aligned}
\Big\|u-\sum_{j=0}^{k+1}w_j\Big\|_{L^2(\widetilde\sigma_{k+1}\Omega\times((1-\widetilde\sigma_{k+1}^2)\tau-\widetilde\sigma_{k+1}^2,\tau))}\leq &
C_0M 4^{-(k+1)(1+\alpha')} \widetilde\sigma_k^{(n+2)/2}\Big (\frac{C}{1-4^{-\alpha'}}\varepsilon_0+\varepsilon_0^{1/2}\Big )\\
\leq &2C_0 M \varepsilon_0^{1/4}\max\Big(\frac{C}{1-4^{-\alpha'}},1\Big)4^{-\frac{(k+2)(n+4+2\alpha')}{2}}.
\end{aligned}
$$
Therefore, by taking $\varepsilon_0>0$ small enough to satisfy $2C_0
\varepsilon_0^{1/4}\max\Big(\frac{C}{1-4^{-\alpha'}},1\Big)\leq 1$,
we have $(4.5)'_{k+1}$.

By the estimate of $\|W-v_{k+1}\|_{L^2(\overline\sigma_0\Omega\times((1-\overline\sigma_0^2)\tau-\overline\sigma_0^2,\tau))}$ and the interior estimate for $\partial_t -\nabla \cdot\overline A\nabla $,
$$
\begin{aligned}
&\|\nabla v_{k+1}\|_{L^{\infty}(\widetilde\sigma_0\Omega\times((1-\widetilde\sigma_0^2)\tau-\widetilde\sigma_0^2,\tau))}\\
&\leq
C_0\|v_{k+1}\|_{L^2(\overline\sigma_0\Omega\times((1-\overline\sigma_0^2)\tau-\overline\sigma_0^2,\tau))}\\
&\leq C_0\Big (\|W\|_{L^2(\overline\sigma_0\Omega\times((1-\overline\sigma_0^2)\tau-\overline\sigma_0^2,\tau))}+\|W-v_{k+1}\|_{L^2(\overline\sigma_0\Omega\times((1-\overline\sigma_0^2)\tau-\overline\sigma_0^2,\tau))}\Big )\\
&\leq C_0M \Bigg [4^{-(k+1)(1+\alpha')}+C_04^{-(k+1)(1+\alpha')}\Big (\frac{C}{1-4^{-\alpha'}}\varepsilon_0+\varepsilon_0^{1/2}\Big )\Bigg]\\
&\leq C_0M 4^{-(k+1)(1+\alpha')}\Bigg [1+2C_0\varepsilon_0^{1/2}\max\Big(\frac{C}{1-4^{-\alpha'}},1\Big)\Bigg].
\end{aligned}
$$
Hence, if we adjust $C$ and $\varepsilon_0$ to satisfy $C \geq C_0 [1+2C_0\varepsilon_0^{1/2}\max(\frac{C}{1-4^{-\alpha'}},1)]$, then we have $(4.4)'_{2,k+1}$. Finally, to see $(4.4)'_{1,k+1}$, we have from the above estimate of $\|v_{k+1}\|_{L^{2}(\widetilde\sigma_0\Omega\times((1-\widetilde\sigma_0^2)\tau-\widetilde\sigma_0^2,\tau))}$ and $(x,t)\in \overline\sigma_{k+1}\Omega\times((1-\overline\sigma_{k+1}^2)\tau-\overline\sigma_{k+1}^2,\tau)$ which is equivalent to $(\widetilde\sigma_k^{-1}x,\widetilde\sigma_k^{-2}t+(1-\widetilde\sigma_k^{-2})\tau)\in \overline\sigma_0\Omega\times((1-\overline\sigma_0^2)\tau-\overline\sigma_0^2,\tau)$. we have
$$
\begin{aligned}
\|w_{k+1}\|_{L^2(\overline\sigma_{k+1}\Omega\times((1-\overline\sigma_{k+1}^2)\tau-\overline\sigma_{k+1}^2,\tau))}\leq &
C_0M 4^{-(k+1)(1+\alpha')}\tilde\sigma_k^{(n+2)/2} \Bigg [1+2C_0\varepsilon_0^{1/2}\max\Big(\frac{C}{1-4^{-\alpha'}},1\Big)\Bigg]\\
\leq & C_0M 4^{-\frac{(k+1)(n+4+2\alpha')}{2}}\Bigg
[1+2C_0\varepsilon_0^{1/2}\max\Big(\frac{C}{1-4^{-\alpha'}},1\Big)\Bigg].
\end{aligned}
$$
Therefore, if we further adjust $C$ and $\varepsilon_0$ to satisfy $C \geq C_0 [1+2C_0\varepsilon_0^{1/2}\max(\frac{C}{1-4^{-\alpha'}},1)]$, then we have $(4.4)'_{1,k+1}$. Thus, we have proven $(4.3)'-(4.5)'$.

Let $C$ be a general constant which is different from the general constant $C$ in $(4.4)'-(4.5)'$. As an easy consequence of $(4.4)'_{2,k}$, we have
$$
\|w_{k}\|_{L^{\infty}(\widetilde\sigma_k\Omega\times((1-\widetilde\sigma_k^2)\tau-\widetilde\sigma_k^2,\tau))}\leq CM4^{-k(1+\alpha')}.\eqno{(4.6)'}
$$
Together with this and $(4.4)'_2$,
$$
\begin{aligned}
\Big | \sum_{j=0}^k w_j(x,t)-\sum_{j=0}^{\infty} w_j(0,\tau)\Big |\leq &
CM\sum_{j=0}^k4^{-j\alpha'}|(x,t-\tau)| + CM \sum_{j=k+1}^{\infty}4^{-j(1+\alpha')}\\
\leq & CM|(x,t-\tau)| +CM4^{-k(1+\alpha')}.
\end{aligned}
$$
Hence, we have
$$
\begin{aligned}
&\Big \| u-\sum_{j=0}^{\infty} w_j(0,\tau)\Big\|_{L^2(\widetilde\sigma_{k}\Omega\times((1-\widetilde\sigma_{k}^2)\tau-\widetilde\sigma_{k}^2,\tau))}\\
&\leq
\Big \| u-\sum_{j=0}^{k} w_j\Big\|_{L^2(\widetilde\sigma_{k}\Omega\times((1-\widetilde\sigma_{k}^2)\tau-\widetilde\sigma_{k}^2,\tau))}+\Big \| \sum_{j=0}^{k} w_j-\sum_{j=0}^{\infty} w_j(0,\tau)\Big\|_{L^2(\widetilde\sigma_{k}\Omega\times((1-\widetilde\sigma_{k}^2)\tau-\widetilde\sigma_{k}^2,\tau))}\\
&\leq
M4^{-\frac{(k+1)(n+4+2\alpha')}{2}}+CM\Bigg[\int_{(1-\widetilde\sigma_{k}^2)\tau-\widetilde\sigma_{k}^2}^{\tau}\int_{\widetilde\sigma_{k}\Omega}\Big(|x|^2+(t-\tau)^2+4^{-2k(1+\alpha')}\Big)dxdt\Bigg]^{1/2}.
\end{aligned}
$$
Hence, by $(t-\tau)^2\leq
\widetilde\sigma_k^4(\tau+1)^2=4^{-4(k+1)}(\tau+1)$, we can absorb
$(t-\tau)^2$ into $4^{-2k(1+\alpha')}$. Since
$\int_{\widetilde\sigma_{k}\Omega}|x|^2dx \leq C4^{-(k+1)(n+2)}$,
$\int_{\widetilde\sigma_{k}\Omega}4^{-2k(1+\alpha')}dx \leq
C4^{-2k(1+\alpha')-(k+1)n}$ and $4^{-2k(1+\alpha')-(k+1)n}\leq
4^{-(k+1)(n+2)}$ for large $k$, we have
$$
\Big \| u-\sum_{j=0}^{\infty} w_j(0,\tau)\Big\|_{L^2(\widetilde\sigma_{k}\Omega\times((1-\widetilde\sigma_{k}^2)\tau-\widetilde\sigma_{k}^2,\tau))}\leq CM 4^{-\frac{(k+1)(n+2)}{2}}\widetilde\sigma_{k}.
$$
Therefore, $u(0,\tau)=\sum_{j=0}^{\infty} w_j(0,\tau)$ and $|\nabla_x u(0,\tau)|\leq CM$ if $\nabla_x u(0,\tau)$ exists.
\end{proof}

\section{Gradient Estimate of Fundamental Solution}
In this section, as we already mentioned in the introduction, we
will give an estimate of $\nabla_x\Gamma(x,t;y,s)$ for a fundamental
solution $\Gamma(x,t;y,s)$ of the operator $\mathcal{L}$ as an
application of our main theorem (Theorem \ref{thm::main}) by
following the argument given in \cite{D-V}. For the readers'
convenience, we repeat the argument.

It is well known that there exists a fundamental solution
$\Gamma(x,t;y,s)$ with the estimate
\begin{equation}\label{eq3.13}
\Gamma(x,t;y,s) \leq \frac{C}{[4\pi
(t-s)]^{n/2}}e^{-\frac{|x-y|^2}{C(t-s)}}\chi_{[s,\infty]}\,\,(t,s\in{\Bbb R},\,t>s,\,\,\mbox{a.e.}\, x, y\in D),\end{equation}
which is positive for $t>s$, where $C>0$ is a constant which depends only on $A$, $n$ and $\chi_{[s,\infty)}$ is the characteristic function of $[s,\infty)$.
(See \cite{A}.)

Now we state the estimate of $\nabla_x\Gamma(x,t;y,s)$.

\begin{proposition}\label{prop3.6} Let $\Gamma(x,t;y,s)$ be the previous fundamental solution of the
operator $\partial_t-\nabla\cdot A\nabla$. There exists
a constant $C>0$ depending only on $A$ and $n$ such that
\begin{equation}\label{eq3.14}
|\nabla_x \Gamma(x,t;y,s)|\leq
\frac{C}{(t-s)^{\frac{n+1}{2}}}e^{-\frac{|x-y|^2}{C(t-s)}},\end{equation}
for any $t,s\in{\Bbb R}$, $t>s$ and almost every $x,y\in D$.
\end{proposition}

\begin{remark}
We recall
that a fundamental solution $G^*(x,t,y,s)$ of the operator $\partial_t+\nabla\cdot A\nabla$ can be given by
\begin{equation}\label{eq3.12}
G^*(x,t;y,s)=\Gamma(y,s;x,t) \quad ((x,t),(y,s)\in Q:=D\times{\Bbb
R},\,(x,t)\neq (y,s)).\end{equation} Hence, estimates similar to
\eqref{eq3.13} and \eqref{eq3.14} hold for $G^*(x,t;y,s)$.
\end{remark}

Before proving Proposition \ref{prop3.6} we give the following
estimate which is necessary for the proof.

\begin{proposition}\label{prop4.2} Let $Q_{\rho}(x_0,t_0)=B_\rho(x_0)\times (t_0-\rho^2,t_0)$, $B_\rho(x_0):=\{x\in{\Bbb R}^n\,;\,|x-x_0|<\rho\}$. There exists a constant $C>0$ depending only on $A$ and $n$ such that the following
inequality holds.
\begin{equation}\label{eq4.9}
\int_{Q_{\rho}(x_0,t_0)}|\Gamma(x,t;\xi,\tau)|^2dxdt\leq
C\frac{\rho^n}{(t_0-\tau)^{n-1}}e^{-\frac{|x_0-\xi|^2}{C(t_0-\tau)}}\,\,(\tau<t_0,\,\mbox{\rm
a.e.}\,\xi\in D),
\end{equation}
where $\rho=\frac{1}{4}[|x_0-\xi|^2+t_0-\tau]^{1/2}$.\end{proposition}
\begin{proof} From the inequality \eqref{eq3.13} we have
\begin{equation}\label{eq4.10}
\int_{Q_{\rho}(x_0,t_0)}|\Gamma(x,t;\xi,\tau)|^2dxdt\leq
C_1\int_{Q_{\rho}(x_0,t_0)}\frac{1}{(t-\tau)^n}e^{-\frac{|x-\xi|^2}{2C_1(t-\tau)}}\chi_{[\tau,+\infty)}dxdt,\end{equation}
where $C_1>0$ is a constant depending only on $A$ and $n$. In what follows we
denote by $I$ the integral at the right-hand side of
\eqref{eq4.10}. We distinguish two cases
\begin{itemize}
\item[i)] $t_0-\rho^2 \le \tau < t_0$,

\item[ii)] $\tau < t_0-\rho^2$.\end{itemize} Let us consider case
i). It is easy to see that there exists a constant
$C>0$ such that
\begin{equation}\label{eq4.11}
C^{-1}\rho \leq |x-\xi| \leq C\rho \quad (x\in
B_{\rho}(x_0)).\end{equation} By \eqref{eq4.11} we have
\begin{equation}\label{eq4.12}
I \leq
c_n\rho^n\int_0^{t_0-\tau}s^{-n}e^{\frac{-\rho^2}{C_2s}}ds,\end{equation}
where $c_n>0$ is a constant depending only on $n$  and
$C_2>0$ is a constant depending only on $A$ and $n$. Now we assume $0 <
t_0-\tau<\frac{\rho^2}{nC_2}$. Since
$s^{-n}e^{-\frac{\rho^2}{C_2s}}$ is an increasing function in
$(0,\frac{\rho^2}{nC_2})$, we have by \eqref{eq4.12}
\begin{equation}\label{eq4.13}
I \leq
\frac{c_n\rho^n}{(t_0-\tau)^{n-1}}e^{-\frac{\rho^2}{C(t_0-\tau)}}.\end{equation}
Further, if we assume $\frac{\rho^2}{nC_2}\leq t_0-\tau \leq
\rho^2$, then
\begin{equation*}
\rho^{-n}I \leq
C_1\int_0^{\rho^2}s^{-n}e^{-\frac{\rho^2}{C_2s}}ds\leq
\frac{C}{(t_0-\tau)^{n-1}}e^{-\frac{\rho^2}{C_2(t_0-\tau)}}
\end{equation*}
due to the equivalence of $t-\tau$ and $\rho^2$, where $c_n$ and $C_2$ are the same kind of constants as before.
Hence, by the last inequality and \eqref{eq4.13}, we have the Proposition
in case i).

Let us consider case ii). It is easy to see that
\begin{equation}\label{eq4.14}
6\rho^2 \leq |x-\xi|^2 + t-\tau \leq 60\rho^2,\end{equation} for
every $(x,t)\in Q_{\rho}(x_0,t_0)$. Moreover, denoting
\begin{equation*}
M_{\rho}=\max\Bigg\{\frac{e^{-\frac{|x-\xi|^2}{2C_1(t-\tau)}}}{(t-\tau)^n}
; (x,t)\in Q_{\rho}(x_0,t_0)\Bigg\}\end{equation*} and taking into
account \eqref{eq4.14}, we have
\begin{equation}\label{eq4.15}
M_{\rho}\leq C\Big(\frac{C_1}{\rho^2}\Big)^n,\end{equation} where
$C>0$ is a constant depending only on $n$. Now, since $\tau<t_0-\rho^2$, we have
\begin{equation}\label{eq4.16}\frac{|x_0-\xi|^2}{t_0-\tau}\leq
16.\end{equation} Therefore, by \eqref{eq4.15} and \eqref{eq4.16},
we have the Proposition in case ii) as well.\end{proof}

\begin{proof}[Proof of Proposition \ref{prop3.6}] By applying our main theorem to the
function $\Gamma(\cdot,\cdot;\xi,\tau)$, we have
\begin{equation}\label{eq4.18}
||\nabla\Gamma(\cdot,\cdot;\xi,\tau)||_{L^{\infty}(Q_{\rho}(x_0,t_0))}\leq
\frac{C}{\rho^{\frac{n+4}{2}}}\Bigg[\int_{Q_{2\rho}(x_0,t_0)}|\Gamma(x,t;\xi,\tau)|^2dxdt\Bigg]^{1/2}.
\end{equation}

Further, applying Proposition \ref{prop4.2} to the right-hand side of
\eqref{eq4.18} we have
\begin{equation*}
||\nabla
\Gamma(\cdot,\cdot;\xi,\tau)||_{L^{\infty}(Q_{\rho}(x_0,t_0))}\leq
\frac{C}{\rho^{\frac{n+4}{2}}}\Bigg[\frac{\rho^n}{(t_0-\tau)^{n-1}}e^{-\frac{|x_0-\xi|^2}{C(t_0-\tau)}}\Bigg]^{1/2}.\end{equation*}
Then, we immediately have \eqref{eq3.14}, because
\begin{equation*}
\frac{1}{\rho^2}\leq \frac{C}{t_0-\tau}.
\end{equation*}
\end{proof}

\section{Appendix A: Construction of Green Function in two Layered Cube}

In this section we will construct the Green function $G^*(x,t;y,s)$ of our operator $\partial_t+\nabla\cdot\overline A\nabla$ in $\sigma_0\Omega\times{\Bbb R}$ with Dirichlet boundary condition on $\partial(\sigma_0\Omega)\times{\Bbb R}$. If $G(x,t;y,s)$ is the Green function of the operator $\mathcal{L}=\partial_t-\nabla\cdot\overline A\nabla$ in $\sigma_0\Omega\times{\Bbb R}$ with Dirichlet boundary condition on $\partial(\sigma_0\Omega)\times{\Bbb R}$, we have
\begin{equation}\label{symmetry}
G^*(x,t;y,s)=G(y,s;x,t).
\end{equation}
Hence, it is enough to construct the Green function $G(x,t;y,s)$.

First we construct a fundamental solution $\Gamma(x,t;y,s)$ of $\mathcal{L}$. We divide the construction into two cases. They are $y_n>0$ and $y_n<0$. We first consider the case $y_n>0$. Let $A,B$ be positive definite symmetric
constant matrices. $A=(a_{ij})_{1\le i,j\le n},
B=(b_{ij})_{1\le i,j\le n}$. Define $\overline A=A+(B-A)\chi_{-}(\xi)$,
where
\begin{equation*} \chi_{-}(\xi)=\begin{cases} 0, \quad \xi_n > 0,\\
1, \quad \xi_n < 0.\end{cases}\end{equation*} Let $\Ga(x,t;y,s)$
be the fundamental solution for
$\p_t-\nabla\cdot(\overline A\nabla_{x})$, that is,
\begin{equation}
\p_t \Ga(x,t;y,s)-\nabla\cdot(\overline A\nabla
\Ga(x,t;y,s))=\delta(x-y)\delta(t-s).\end{equation} Note that
$\Gamma(x,t;y,s)$ is also the fundamental of the Cauchy problem at
$t=s$ for the operator $\p_t-\nabla\cdot\overline A\nabla$. Let
$\hG$ be the Laplace transform of $\Ga$ with respect to $t$, that
is,
\begin{equation}\hG (x,\tau;y,s) = \int_0^{\infty}
e^{-t\tau} \Ga (x,t;y,s)dt.\end{equation} Then $\hG$ satisfies
\begin{equation}
\tau
\hG(x,t)-\nabla\cdot(\overline A\nabla_{x}\hG(x,t))=\delta(x-y)e^{-\tau
s}.\end{equation} Now, we denote $\Ga$ for different regions as
follows:
\begin{equation}\Ga=\begin{cases}\Ga^{11} \quad \mbox{for}\quad
x_n > y_n,\\
\Ga^{12} \quad \mbox{for} \quad y_n > x_n > 0,\\
\Ga^2 \quad \mbox{for} \quad 0 > x_n.\end{cases}\end{equation} For
$\vp\in C_0^{\infty}(\RR^n)$, we have
\begin{align*}
0 & =\int_{\RR^n} [ \tau\hG\vp+\overline
A\nabla\hG\cdot\nabla\vp-e^{\tau
s}\delta(x-y)\vp]dx\\
& = \int_{\RR^n}[\tau\hG\vp-e^{-\tau s}\delta(x-y)\vp]dx\\
&  \quad +
\int_{x_n>y_n}A\nabla\hG^{11}\cdot\nabla\vp+\int_{y_n>x_n>0}A\nabla\hG^{12}\cdot\nabla
\vp+\int_{0>x_n}B\nabla\hG^{2}\cdot\nabla\vp\\
& = \int_{\RR^n}\tau\hG\vp dx-\int_{x_n=y_n}e^{-\tau
s}\delta(x'-y')\vp dx'-\int_{x_n=y_n}A\nabla\hG^{11}\cdot e_n \vp
dx'-\int_{x_n> y_n} \nabla\cdot(A\nabla\hG^{11})\vp dx \\
& \quad + \int_{x_n=y_n}A\nabla \hG^{12}\cdot e_n \vp
dx'-\int_{x_n=0}A\nabla \hG^{12}\cdot e_n\vp
dx'-\int_{y_n>x_n>0}\nabla\cdot(A\nabla\hG^{12})\vp\\
& \quad +\int_{x_n=0}B\nabla\hG^{2}\cdot
e_n \vp dx'-\int_{0> x_n}\nabla\cdot (B\nabla\hG^{2})\vp\\
& = \int_{\RR^n}\tau \hG \vp dx - \int_{x_n>y_n}\nabla(A\cdot
\nabla\hG^{11})\vp-\int_{y_n>x_n>0}\nabla\cdot(A\nabla\hG^{12})\vp
-\int_{0>x_n}\nabla\cdot (B\nabla\hG^{2})\vp\\
& \quad + \int_{x_n=y_n}[-e^{-\tau
s}\delta(x'-y')-A\nabla\hG^{11}\cdot e_n + A\nabla\hG^{12}\cdot
e_n]\vp dx'+\int_{x_n=0}[-A\nabla\hG^{12}\cdot e_n+
B\nabla\hG^{2}\cdot e_n]\vp dx'.
\end{align*} Therefore, we have the following transmission
problem
\begin{equation*}\begin{cases}
\nabla\cdot(A\nabla\hG^{11})-\tau \hG^{11}=0 \quad \mbox{in} \quad
x_n>y_n\\
\nabla\cdot(A\nabla\hG^{12})-\tau \hG^{12}=0 \quad \mbox{in} \quad
y_n>x_n>0\\
\nabla\cdot(B\nabla\hG^{2})-\tau \hG^{2}=0 \quad \mbox{in} \quad
0>x_n\\
\hG^{11}-\hG^{12}=0 \quad \mbox{on} \quad x_n=y_n\\
A\nabla(\hG^{11}-\hG^{12})\cdot e_n=-e^{-\tau s}\delta(x'-y')
\quad \mbox{on} \quad x_n=y_n\\
\hG^{12}-\hG^{2}=0 \quad \mbox{on} \quad x_n=0\\
A\nabla\hG^{12}\cdot e_n-B\nabla\hG^{2}\cdot e_n=0 \quad \mbox{on}
\quad x_n=0,\end{cases}\end{equation*} where $e_n=(0,\cdots,0,1)$.
Let $\phi^{11,12,2}$ be the Fourier transforms of $\hG^{11,12,2}$
for $x'=(x_1,\cdots,x_{n-1})$. From now on, we use
$\xi'=(\xi_1,\cdots,\xi_{n-1})$ to denote the Fourier variable
associated with $x'$. Then, we have
\begin{equation}\label{linear-system}\begin{cases}\ds a_{nn}\frac{\p^2 \phi^{11}}{\p
x_n^2}+2i\Big(\sum_{j=1}^{n-1}a_{jn}\xi_j\Big)\frac{\p \phi^{11}}{\p
x_n}-(\tA\xi'\cdot\xi' + \tau) \phi^{11} = 0  \hfill\mbox{in } \{x_n > y_n\},\\
\ds a_{nn}\frac{\p^2 \phi^{12}}{\p
x_n^2}+2i\Big(\sum_{j=1}^{n-1}a_{jn}\xi_j\Big)\frac{\p \phi^{12}}{\p
x_n}-(\tA\xi'\cdot\xi' + \tau) \phi^{12} = 0  \hfill\mbox{in } \{y_n > x_n > 0\},\\
\ds b_{nn}\frac{\p^2\phi^{2}}{\p
x_n^2}+2i\Big(\sum_{j=1}^{n-1}b_{jn}\xi_j\Big)\frac{\p \phi^{2}}{\p
x_n}-(\tB\xi'\cdot\xi' + \tau) \phi^{2} = 0  \hfill \mbox{in } \{x_n < 0 \},\\
\ds \phi^{11} - \phi^{12} = 0  \hfill \mbox{on } \{x_n=y_n\},\\
\ds a_{nn}\Big(\frac{\p \phi^{11}}{\p x_n} - \frac{\p \phi^{12}}{\p
x_n}\Big) = -e^{-\tau s}e^{-i y'\cdot \xi'}  \hfill
\mbox{on } \{x_n=y_n\},\\
\ds \phi^{12} - \phi^{2} = 0 \hfill\mbox{on } \{x_n=0\},\\
\ds a_{nn}\frac{\p \phi^{12}}{\p x_n}
+i\Big(\sum_{j=1}^{n-1}a_{jn}\xi_j\Big)\phi^{12}- b_{nn}\frac{\p
\phi^{2}}{\p x_n}-i\Big(\sum_{j=1}^{n-1}b_{jn}\xi_j\Big)\phi^{2} =
0 \quad \hfill \mbox{on }\{x_n=0\},\end{cases}\end{equation} where
$\tA=(a_{ij})_{1\le i,j\le n-1},\ \tB=(b_{ij})_{1\le i,j\le n-1}$.
In addition, we put another conditions
\begin{equation}
\lim_{x_n\rightarrow \infty}\phi^{11}=0,\quad \lim_{x_n\rightarrow
-\infty}\phi^2=0.\end{equation} For simplicity of notations, let
us put
\begin{equation*}\begin{cases}
a=\sum_{j=1}^{n-1}a_{jn}\xi_j, \quad
b=\sum_{j=1}^{n-1}b_{jn}\xi_j,\\
\Theta_A=[a_{nn}(\tA\xi'\cdot\xi' + \tau)-a^2]^{1/2}, \quad
\Theta_B=[b_{nn}(\tB\xi'\cdot\xi' +
\tau)-b^2]^{1/2},\end{cases}\end{equation*} where the real parts of
$\Theta_A$ and $\Theta _B$ are positive. From the first three
differential equations in \eqref{linear-system}, we have
\begin{align*}
\phi^{11} & = C_1 \exp\Big[\frac{-ia-\Theta_A}{a_{nn}}x_n\Big],\\
\phi^{12} & = C_2 \exp\Big[\frac{-ia-\Theta_A}{a_{nn}}x_n\Big]+C_3
\exp\Big[\frac{-ia+\Theta_A}{a_{nn}}x_n\Big],\\
\phi^{2} & = C_4
\exp\Big[\frac{-ib+\Theta_B}{b_{nn}}x_n\Big].\end{align*} Conditions
on $x_n=y_n$ and $x_n=0$ imply that
\begin{equation*}\begin{cases}
C_1-C_2-C_3\exp\Big[\frac{2\Theta_A}{a_{nn}}y_n\Big]=0,\\
(C_1-C_2)(-ia-\Theta_A)-C_3(-ia+\Theta_A)\exp\Big[\frac{2\Theta_A}{a_{nn}}y_n\Big]=-e^{-\tau
s-iy'\cdot\xi'}\exp\Big[\frac{ia+\Theta_A}{a_{nn}}y_n\Big],\\
C_2+C_3-C_4=0,\\
C_2\Theta_A-C_3\Theta_A+C_4\Theta_B=0.\end{cases}\end{equation*}
Then, we have
\begin{align*}
C_1 & = \frac{1}{2\Theta_A}e^{-\tau
s-iy'\cdot\xi'}\exp\Big[\frac{ia+\Theta_A}{a_{nn}}y_n\Big]+
\frac{\Theta_A-\Theta_B}{2\Theta_A(\Theta_A+\Theta_B)}e^{-\tau
s-iy'\cdot\xi'}\exp\Big[\frac{ia-\Theta_A}{a_{nn}}y_n\Big],\\
C_2 & =
\frac{\Theta_A-\Theta_B}{2\Theta_A(\Theta_A+\Theta_B)}e^{-\tau
s-iy'\cdot\xi'}\exp\Big[\frac{ia-\Theta_A}{a_{nn}}y_n\Big],\\
C_3 & = \frac{1}{2\Theta_A}e^{-\tau
s-iy'\cdot\xi'}\exp\Big[\frac{ia-\Theta_A}{a_{nn}}y_n\Big],\\
C_4 & = \frac{1}{\Theta_A+\Theta_B}e^{-\tau
s-iy'\cdot\xi'}\exp\Big[\frac{ia-\Theta_A}{a_{nn}}y_n\Big].\end{align*}
Hence, we have
\begin{align*}
\phi^{11} & = \frac{1}{2\Theta_A}e^{-\tau
s-iy'\cdot\xi'}\exp\Big[\frac{-ia-\Theta_A}{a_{nn}}x_n+\frac{ia+\Theta_A}{a_{nn}}y_n\Big]\\
& \quad
+\frac{\Theta_A-\Theta_B}{2\Theta_A(\Theta_A+\Theta_B)}e^{-\tau
s-iy'\cdot\xi'}\exp\Big[\frac{-ia-\Theta_A}{a_{nn}}x_n+\frac{ia-\Theta_A}{a_{nn}}y_n\Big],\\
\phi^{12} & =
\frac{\Theta_A-\Theta_B}{2\Theta_A(\Theta_A+\Theta_B)}e^{-\tau
s-iy'\cdot\xi'}\exp\Big[\frac{-ia-\Theta_A}{a_{nn}}x_n+\frac{ia-\Theta_A}{a_{nn}}y_n\Big]\\
& \quad +\frac{1}{2\Theta_A}e^{-\tau s-iy'\cdot\xi'}
\exp\Big[\frac{-ia+\Theta_A}{a_{nn}}x_n+\frac{ia-\Theta_A}{a_{nn}}y_n\Big],\\
\phi^{2} & =  \frac{1}{\Theta_A+\Theta_B}e^{-\tau s-iy'\cdot\xi'}
\exp\Big[\frac{-ib+\Theta_B}{b_{nn}}x_n+\frac{ia-\Theta_A}{a_{nn}}y_n\Big].\end{align*}
Therefore, we have the
following forms for $\Ga$
\begin{align*}
\ds
\Ga^{11}(x,t;y,s)&=\frac{1}{(2\pi)^{n}i}\int_{\RR^2}e^{i(x'-y')\cdot\xi'}
\int_{\sigma-i\infty}^{\sigma+i\infty}e^{\tau
(t-s)}V^{11}(x_n,y_n,\xi',\tau)d\tau d\xi'\\
\Ga^{12}(x,t;y,s)&=\frac{1}{(2\pi)^{n}i}\int_{\RR^2}e^{i(x'-y')\cdot\xi'}
\int_{\sigma-i\infty}^{\sigma+i\infty}e^{\tau
(t-s)}V^{12}(x_n,y_n,\xi',\tau)d\tau d\xi'\\
\Ga^{2}(x,t;y,s)&=\frac{1}{(2\pi)^{n}i}\int_{\RR^2}e^{i(x'-y')\cdot\xi'}
\int_{\sigma-i\infty}^{\sigma+i\infty}e^{\tau
(t-s)}V^{2}(x_n,y_n,\xi',\tau)d\tau d\xi',\end{align*} where $\sigma
> 0$ and
\begin{align*}
V^{11}(x_n,y_n,\xi',\tau) & =\frac{1}{2\Theta_A}\exp\Big[\frac{-ia-\Theta_A}{a_{nn}}x_n+\frac{ia+\Theta_A}{a_{nn}}y_n\Big]\\
& \quad
+\frac{\Theta_A-\Theta_B}{2\Theta_A(\Theta_A+\Theta_B)}\exp\Big[\frac{-ia-\Theta_A}{a_{nn}}x_n+\frac{ia-\Theta_A}{a_{nn}}y_n\Big],\\
V^{12}(x_n,y_n,\xi',\tau) & =
\frac{\Theta_A-\Theta_B}{2\Theta_A(\Theta_A+\Theta_B)}\exp\Big[\frac{-ia-\Theta_A}{a_{nn}}x_n+\frac{ia-\Theta_A}{a_{nn}}y_n\Big]\\
& \quad +\frac{1}{2\Theta_A}
\exp\Big[\frac{-ia+\Theta_A}{a_{nn}}x_n+\frac{ia-\Theta_A}{a_{nn}}y_n\Big],\\
V^{2}(x_n,y_n,\xi',\tau) & =  \frac{1}{\Theta_A+\Theta_B}
\exp\Big[\frac{-ib+\Theta_B}{b_{nn}}x_n+\frac{ia-\Theta_A}{a_{nn}}y_n\Big].\end{align*}

Next we consider the case $y_n<0$.
We denote $\Ga$ for different regions as follows. That is
\begin{equation}\Ga=\begin{cases}\Ga^{1} \quad \mbox{for}\quad
x_n > 0,\\
\Ga^{21} \quad \mbox{for} \quad 0 > x_n > y_n,\\
\Ga^{22} \quad \mbox{for} \quad y_n >
x_n.\end{cases}\end{equation}
Let $\hG^{1,21,22}$ be the Laplace transform of $\Gamma^{1,21,22}$ with respect to $t$ and
$\phi^{1,21,22}$ be the Fourier transforms of $\hG^{1,21,22}$ with respect to
$x'=(x_1,\cdots,x_{n-1})$. Here we used the notation $\Gamma^{1,21,22}$ for example to represent one of $\Gamma^1,\,\Gamma^{21},\,\Gamma^{22}$.

Then, by a similar argument as we did for the case $y_n>0$, we have
\begin{align*}
\phi^{1} & = \frac{1}{\Theta_A+\Theta_B}e^{-\tau
s-iy'\cdot\xi'}\exp\Big[\frac{-ia-\Theta_A}{a_{nn}}x_n+\frac{ib+\Theta_B}{b_{nn}}y_n\Big],\\
\phi^{21} & =\frac{1}{2\Theta_B}e^{-\tau s-iy'\cdot\xi'}
\exp\Big[\frac{-ib-\Theta_B}{b_{nn}}x_n+\frac{ib+\Theta_B}{b_{nn}}y_n\Big]
\\
& \quad
+\frac{\Theta_B-\Theta_A}{2\Theta_B(\Theta_B+\Theta_A)}e^{-\tau
s-iy'\cdot\xi'}\exp\Big[\frac{-ib+\Theta_B}{b_{nn}}x_n+\frac{ib+\Theta_B}{b_{nn}}y_n\Big],\\
\phi^{22} & =
\frac{\Theta_B-\Theta_A}{2\Theta_B(\Theta_B+\Theta_A)}e^{-\tau
s-iy'\cdot\xi'}\exp\Big[\frac{-ib+\Theta_B}{b_{nn}}x_n+\frac{ib+\Theta_B}{b_{nn}}y_n\Big]\\
& \quad +\frac{1}{2\Theta_B}e^{-\tau s-iy'\cdot\xi'}
\exp\Big[\frac{-ib+\Theta_B}{b_{nn}}x_n+\frac{ib-\Theta_B}{b_{nn}}y_n\Big],\end{align*}
Therefore, we have the
following forms for $\Ga$
\begin{align*}
\ds
\Ga^{1}(x,t;y,s)&=\frac{1}{(2\pi)^{n}i}\int_{\RR^2}e^{i(x'-y')\cdot\xi'}
\int_{\sigma-i\infty}^{\sigma+i\infty}e^{\tau
(t-s)}V^{1}(x_n,y_n,\xi',\tau)d\tau d\xi'\\
\Ga^{21}(x,t;y,s)&=\frac{1}{(2\pi)^{n}i}\int_{\RR^2}e^{i(x'-y')\cdot\xi'}
\int_{\sigma-i\infty}^{\sigma+i\infty}e^{\tau
(t-s)}V^{21}(x_n,y_n,\xi',\tau)d\tau d\xi'\\
\Ga^{22}(x,t;y,s)&=\frac{1}{(2\pi)^{n}i}\int_{\RR^2}e^{i(x'-y')\cdot\xi'}
\int_{\sigma-i\infty}^{\sigma+i\infty}e^{\tau
(t-s)}V^{22}(x_n,y_n,\xi',\tau)d\tau d\xi',\end{align*} where
$\sigma
> 0$ and
\begin{align*}
V^{1}(x_n,y_n,\xi',\tau) & =\frac{1}{\Theta_A+\Theta_B}\exp\Big[\frac{-ia-\Theta_A}{a_{nn}}x_n+\frac{ib+\Theta_B}{b_{nn}}y_n\Big],\\
V^{21}(x_n,y_n,\xi',\tau) & = \frac{1}{2\Theta_B}
\exp\Big[\frac{-ib-\Theta_B}{b_{nn}}x_n+\frac{ib+\Theta_B}{b_{nn}}y_n\Big]\\
& \quad
+\frac{\Theta_B-\Theta_A}{2\Theta_B(\Theta_B+\Theta_A)}\exp\Big[\frac{-ib+\Theta_B}{b_{nn}}x_n+\frac{ib+\Theta_B}{b_{nn}}y_n\Big],\\
V^{22}(x_n,y_n,\xi',\tau) & =
\frac{\Theta_B-\Theta_A}{2\Theta_B(\Theta_B+\Theta_A)}\exp\Big[\frac{-ib+\Theta_B}{b_{nn}}x_n+\frac{ib+\Theta_B}{b_{nn}}y_n\Big]\\
& \quad +\frac{1}{2\Theta_B}
\exp\Big[\frac{-ib+\Theta_B}{b_{nn}}x_n+\frac{ib-\Theta_B}{b_{nn}}y_n\Big].\end{align*}

\medskip
Next we will show how to construct the Green function $G(x,t;y,s)$ from $\Gamma(x,t;y,s)$ by using the argument given in \cite{R}. For example, consider a face $x_1=-\sigma_0$ of
$\sigma_0\Omega$. For the simplicity of notations, we introduce $\tilde\Gamma(x_1,x",t;y,s)=\Gamma(x_1-\sigma_0,x",t;y_1-\sigma_0,y",s)$ with $x"=(x_2,\cdots,x_n)$. Then, $\tilde\Gamma$ solves
\begin{equation}
\left\{\begin{array}{l}
(\partial_t-\nabla\cdot\overline A\nabla)\tilde\Gamma=0\quad\mbox{\rm in}\quad{\Bbb R}^n\times(s,\infty)\\
\lim_{t\downarrow s}\int_{{\Bbb R}^n} \tilde
\Gamma(x,t;y,s)\phi(y)dy=\phi(x_1-\sigma_0,x")\,\,(\phi\in
C_0^\infty({\Bbb R}^n))
\end{array}
\right.
\end{equation}
Let us distinguish $\overline A$ here by denoting it by $\tilde A$. Now, we extend $\tilde A=(a_{ij}^+)$ in $x_1>0$ denoted by $\tilde A_+$ to $x_1<0$ as follows.
That is we define $\tilde A_-=(a_{ij}^-)$ by $a_{11}^-=a_{11}^+$, $a_{ij}^-=a_{ij}^+\,(2\le i,j\le n)$, $a_{1j}^-=-a_{1j}^+\,(2\le j\le n)$.
Then, if we define $\tilde\Gamma'(x,t;y,s)$ by $\tilde\Gamma'(x_1,x",t;y,s)=\tilde\Gamma(\pm x_1,x",t;y,s)\,(\pm x_1>0)$, then $\tilde \Gamma'$ satisfies
\begin{equation}
\left\{\begin{array}{l}
(\partial_t-\nabla\cdot\tilde A\nabla)\tilde\Gamma'=0\quad\mbox{\rm in}\quad{\Bbb R}^n\times(s,\infty)\\
\lim_{t\downarrow s}\int_{{\Bbb R}^n} \tilde
\Gamma'(x,t;y,s)\phi(y)dy=\phi(x_1-\sigma_0,x")\,\,(\phi\in
C_0^\infty(\{x_1>0\}\times{\Bbb R}^{n-1})
\end{array}
\right.
\end{equation}
and $\tilde\Gamma'(x,t;-y,s)=\tilde\Gamma'(-x_1,x",t;y,s)\,\,(t>s,\,\mbox{\rm a.e.}\,x,y\in{\Bbb R}^n)$. Hence, $\tilde\Gamma'(x,t;y,s)-\tilde\Gamma'(x,t;-y_1,y",s)$ for $x_1,\,y_1>0$ is the Green function in the domain $\{x_1>0\}$ satisfying the  Dirichlet boundary condition on $x_1=0$.
Repeating this argument for other faces of $\sigma_0\Omega$, we can construct the Green function $G(x,t;y,s)$ for $x,\,y\in\sigma_0\Omega$.
It is clear from its construction that $G(x,t;y,s)$ satisfies the estimate \eqref{eqn::esGreen} and by \eqref{symmetry}, $G^*(x,t;y,s)$
also satisfies the same estimate.

\section{Appendix B: Estimate of the Green Function}

In order to give its meaning to the fundamental solution $\Gamma(x,t;y,s)$ constructed in the previous section and estimate it, we need the following theorem.
\begin{theorem}[Lemmas 2 and 3 in \cite{Arima}]\label{Arima} For each $\rho\geq 0$,
let $g(\xi',\eta;\rho)$ be holomorphic function of $(\xi',\eta)$
in $L_{\mu}^{n-1} \subset \mathbb{C}^{n-1}\times \mathbb{C}$ for
some $\mu>0$ where
\begin{equation*}
L_{\mu}^{n-1}=\{(\xi',\eta)\in
\mathbb{C}^{n-1}\times\mathbb{C};Im\ \eta < \mu(|Re\ \eta|+|Re\
\xi'|^2)-\mu^{-1}|Im\ \xi'|^2\}.\end{equation*} We assume the following estimate for $g(\xi',\eta;\rho)$. That is, there exist some constants $C>0$ and $c>0$ such that
\begin{equation}\label{estimate of density}
|g(\xi',\eta;\rho)|\leq C(|\xi'|+|\eta|^{1/2})^l \exp(-c\rho
(|\xi'|+|\eta|^{1/2}))\exp(C|\mbox{\rm Im}\xi'|)
\end{equation} for $(\xi',\eta)\in
L_{\mu}^{n-1}$, $l<0$, and $\rho\geq 0$. Then
\begin{equation}
|G(x',t;\rho)|\leq
Ct^{-\frac{n-1}{2}-\frac{l}{2}-1}\exp\Big[-c\frac{|x'|^2+{\rho}^2}{t}\Big],
\end{equation}
where we set
\begin{equation*}
G(x',t;\rho)=(2\pi)^{-n}\int_{\RR^{n-1}}e^{ix'\cdot\xi'}\int_{-\infty-iq}^{\infty-iq}e^{it\eta}g(\xi',\eta;\rho)d\eta
d\xi'\end{equation*}
with an arbitrarily fixed positive number $q$ for
$\rho>0$ and
\begin{equation*}
G(x',t;0)\equiv \lim_{\rho\downarrow 0}G(x',t;\rho).
\end{equation*}
\end{theorem}

\begin{remark}${}$
\newline
(i)The theorem still holds even in the case the amplitude $g$ depends
on $(x_n,y_n)$.
\newline
(ii) In \cite{Arima} the factor $\exp(C|\mbox{\rm Im}\xi'|)$ does not exist in the estimate \eqref{estimate of density}. However, a slight modification of the proof given in \cite{Arima} can include the case when
the factor exists in the estimate. This was pointed out by Dr. S. Nagayasu.

\end{remark}

We will apply Theorem \ref{Arima} to estimate
$\Gamma^{11},\,\Gamma^{12},\,\Gamma^2$ in Appendix A. For this, we
have to show that the assumptions of Theorem \ref{Arima} are
satisfied in our case. To begin with we show that
$V^{11}(x_n,y_n,\xi',\eta)$, $V^{12}(x_n,y_n,\xi',\eta)$,
$V^2(x_n,y_n,\xi',\eta)$ are holomorphic in $(\xi',\eta)\in
L_\mu^{n-1}$ uniformly for $(x_n,y_n)$ with those $x_n,y_n$
satisfying the conditions attached to the definitions of
$V^{11},\,V^{12},\,V^2$. For simplicity we refer this property as
{\it uniform analyticity in $L_\mu^{n-1}$} of $V$. Let
$\gamma=(\gamma_{ij})$ be either $A$ or $B$. Then, the
characteristic equation of each equation of \eqref{linear-system}
with respect to $\lambda$ has the form in terms of $\xi_n=i\lambda$
\begin{equation*}
p_0\xi_n^2 +
p_1(\xi')\xi_n+(p_2(\xi')-\tau)=0,\end{equation*} where
$p_0=-\gamma_{nn}$,
$p_1(\xi')=-2\sum_{j=1}^{n-1}\gamma_{nj}\xi_j$,
$p_2(\xi')=-\sum_{i,j=1}^{n-1}\gamma_{ij}\xi_i\xi_j$ and
$\tau=i\eta$. Then, the roots are given as $\ds \xi_n=\frac{-p_1
\pm z^{\pm}}{2p_0}$, where $z^{\pm}=\sqrt{p_1^2-4p_0(p_2-\tau)}$
and $\pm Im\ z^{\pm}>0$. By the ellipticity, for some constant
$c'>0$, we have
\begin{equation*}
p_1^2-4p_0 p_2<-c'|\xi'|^2 \quad (\xi'\in\RR^2\setminus\{0\},\
x'\in U),\end{equation*} where $U$ is an bounded open set in which
$p_0,p_1,p_2$ are smooth.

From the construction of $V^{11},\,V^{12},\,V^2$, the uniform
analyticity of $V$ in $L_\mu^{n-1}$ easily follows from the
following lemma.

\begin{lemma}\label{analyticity}
There exists $\mu>0$ such that $p_1^2-4p_0 p_2 +4p_0 i\eta\not\in
[0,\infty)$ for $(\xi',\eta)\in L_{\mu}^{n-1}$ and $(x_n,y_n)$ with
those $x_n,y_n$ satisfying the condition attached to the definition
of $V^l$.
\end{lemma}

\bigskip

We will prove Lemma \ref{analyticity} by a contradiction argument. We first note
that
\begin{equation}\label{remind}
L_{\mu}^{n-1}\ni(\xi',\eta)\Longleftrightarrow Im\ \eta < \mu(|Re\
\eta|+|Re\ \xi'|^2)-\mu^{-1}|Im\ \xi'|^2.\end{equation} Suppose that
for $(\xi',\eta)\in L_{\mu}^{n-1},$ there is $m \geq 0$ such that
$p_1^2-4p_0 p_2 +4p_0 i\eta=m$. Put $\alpha=-\gamma_{nn}$,
$\beta=(\beta_1,\cdots,\beta_{n-1}):=-2(\gamma_{n1},\cdots,\gamma_{nn-1})$
and $\widetilde{\gamma}=(\widetilde{\ga_{ij}})_{1\leq i,j\leq
n-1}:=(-\gamma_{ij})_{1\leq i,j\leq n-1}$. Then, we have
\begin{align*}
& p_0=\a<0,\ p_1(\xi')=\beta\cdot\xi'\\
&p_2(\xi')=(\widetilde{\gamma}\xi')\cdot\xi'<0\quad
\mbox{for}\quad \xi'\in\RR^{n-1}\setminus\{0\}\\
& p_1^2=\sum_{j,k=1}^2\beta_j \beta_k \xi_j
\xi_k=(\beta\otimes\beta):(\xi'\otimes\xi')\\
& m=(\beta\otimes\beta):(\xi\otimes\xi')-4\a
(\widetilde{\gamma}\xi')\cdot\xi'+4i\a\eta.\end{align*} For simplicity, we denote
$\xi'_R=Re\ \xi'$ and $\eta_R=Re\ \eta$ etc. Then, we have
\begin{align*}
m & =
(\beta\otimes\beta):(\xi'_R\otimes\xi'_R)-(\beta\otimes\beta):(\xi'_I\otimes\xi'_I)
+2i(\beta\otimes\beta):(\xi'_R\otimes\xi'_I)\\
&\quad
-4\a\{(\widetilde{\gamma}\xi'_R)\cdot\xi'_R-(\widetilde{\gamma}\xi'_I)\cdot\xi'_I\}-8i\a(\widetilde{\gamma}\xi'_R)\cdot\xi'_I+4i\a\eta.\end{align*}
That is, we have
\begin{equation}\label{system}
\begin{cases}
(\beta\otimes\beta):(\xi'_R\otimes\xi'_R)-(\beta\otimes\beta):(\xi'_I\otimes\xi'_I)
-4\a\{(\widetilde{\gamma}\xi'_R)\cdot\xi'_R-(\widetilde{\gamma}\xi'_I)\cdot\xi'_I\}-4\a\eta_I=m\\
(\beta\otimes\beta):(\xi'_R\otimes\xi'_I)-4\a(\widetilde{\gamma}\xi'_R)\cdot\xi'_I+2\a\eta_R=0.\end{cases}\end{equation}
From the first equation, we have
\begin{equation}\label{cl}
\{-(\beta\otimes\beta):(\xi'_R\otimes\xi'_R)+4\a(\widetilde{\gamma}\xi'_R)\cdot\xi'_R\}+\{(\beta\otimes\beta):(\xi'_I\otimes\xi'_I)
-4\a(\widetilde{\gamma}\xi'_I)\cdot\xi'_I\}=-m-4\a\eta_I.
\end{equation} The left
hand side (LHS) of \eqref{cl} has the estimate $\mbox{LHS}>
c'|\xi'_R|^2-c''|\xi'_I|^2$ for some positive constants $c'$ and
$c''$. For the right hand side (RHS) of \eqref{cl}, by the
definition of $L_{\mu}^{n-1}$, we have from the second equation in
\eqref{system}
\begin{align*}
\mbox{RHS} &\leq -4\a\eta_I
<(-4\a)\{\mu(|\eta_R|+|\xi'_R|^2)-\mu^{-1}|\xi'_I|^2\}\\
&=(-4\a)\{\mu((-2\a)^{-1}|(\beta\otimes\beta):(\xi'_R\otimes\xi'_R)-4\a(\widetilde{\gamma}\xi'_R)\cdot\xi'_I|+|\xi'_R|^2)-\mu^{-1}|\xi'_I|^2\}\\
&\leq \mu
M(|\xi'_R|^2+|\xi'_I|^2)-(-4\a)\mu^{-1}|\xi'_I|^2\end{align*} for
some positive constant $M$. Thus, we have
\begin{equation*}
\tilde{c}(|\xi'_R|^2+|\xi'_I|^2)\leq \mu M
(|\xi'_R|^2+|\xi'_I|^2), \quad (\mu>0).\end{equation*} By taking
$\mu>0$ small, we have $\xi'_R=\xi'_I=0$ and hence $\eta_R=0$ by
the second equation in \eqref{system}. Then, $-4\a\eta_I=m$ gives
$\eta_I\geq 0$. This contradicts to \eqref{remind}.

Next, we show the type of estimate \eqref{estimate of density}. Let
$V$ be one of $V^l$, $l=11,12,2$ for either the case $y_n>0$ or $y_n<0$. Then, there exist positive constants $c,\,C$ such that
\begin{equation}\label{estimate of hG}
|V(x_n,y_n,\xi',\eta)|\le C
(|\xi'|+|\eta|^{1/2})^{-1}\exp\Big(-c|x_n-y_n|(|\xi'|+|\eta|^{1/2})\Big)\exp\Big(C|\mbox{\rm Im}\xi'||x_n-y_n|\Big)
\end{equation}
uniformly for $(x_n,y_n)$ with those $x_n,y_n$ satisfying the
condition attached to the definition of $V^l$. Hence, by Theorem
\ref{Arima}, there exist a constant $C>0$ such that
\begin{equation*}
|\Gamma(x,t;y,s)|\le C
(t-s)^{-\frac{n}{2}}\exp\Big(-c'\frac{|x'-y'|^2+(x_n-y_n)^2}{t-s}\Big)\,\,(t>s).\end{equation*}

For the gradient estimate of $\Gamma(x,t;y,s)$ we argue as follows. If we formally differentiate $\Gamma(x,t;y,s)$ by $\partial_{x_j}\,(1\le
j\le n-1)$, then the integrand is multiplied by $i\xi_j$. Then, multiply what we got by $x_k-y_k\,(1\le k\le n-1)$ and then integrate by parts with respect to $\xi_k$. By these procedures, we end up with an integrand which satisfies the same type of estimate as \eqref{estimate of hG}. Also, if we multiply by $x_n-y_n$ instead of multiplying by $x_j-y_j\,(1\le j\le n-1)$, we also have the same type of estimate as \eqref{estimate of hG}, because
$(x_n-y_n)(|\xi'|+|\eta|^{1/2})\exp\Big(-c|x_n-y_n|(|\xi'|+|\eta|^{1/2})$ is bounded. Hence, we have
\begin{equation*}|\nabla_{x'}\Gamma(x,t;y,s)|\le C
(t-s)^{-\frac{n+1}{2}}\exp\Big(-c'\frac{|x-y|^2}{t-s}\Big)\,\,(t>s).
\end{equation*}
We can handle the derivative $\partial_{x_n}\Gamma(x,t;y,s)$ in a similar way. Therefore, we have
\begin{equation}
|\nabla_x\Gamma(x,t;y,s)|\le C
(t-s)^{-\frac{n+1}{2}}\exp\Big(-c'\frac{|x-y|^2}{t-s}\Big)\,\,(t>s).
\end{equation}

\bigskip
\noindent
{\bf Acknowledgement} The authors thank Dr. Sei Nagayasu for the useful discussions.

\bibliographystyle{plain}

%%%%%%%%%%%%%%%%%%%%%%%%%%%%%%%%%%%%

\label{lastpage}

\end{document}